\documentclass[12pt]{article}
\usepackage[utf8]{inputenc}
\usepackage{amsfonts}
\usepackage{scrextend}
\usepackage{mathtools}
\usepackage{float}
\usepackage{amsmath}
\usepackage[margin=1.2in]{geometry}
\usepackage{authblk}
\usepackage{amssymb}
\usepackage{epigraph}
\usepackage[english]{babel}
\usepackage[nottoc]{tocbibind}
\usepackage{tikz-cd}
\usepackage{graphicx}
\usetikzlibrary{matrix}
\usepackage{amsthm}
\usepackage[mathscr]{euscript}
 \let\mathscr\relax% just so we can load this and rsfs
\usepackage[scr]{rsfso}
\newcommand{\Mod}[1]{\ (\mathrm{mod}\ #1)}
\newtheorem{theorem}{Theorem}[section]

\newtheorem*{defi}{Definitions}
\newtheorem{corollary}[theorem]{Corollary}
\newtheorem{lemma}[theorem]{Lemma}

\title{A New Weak Choice Principle}
\date{\vspace{-5ex}}
\author{Lorenz Halbeisen, Riccardo Plati, Salome Schumacher}
\providecommand{\keywords}[1]{\textit{key-words:} #1}
\providecommand{\mathclass}[1]{\textit{2010 Mathematics Subject Classification:} #1}
\affil{Department of Mathematics, ETH Z{\"u}rich}

\begin{document}

\maketitle

\begin{abstract}
\noindent For every natural number $n$ we introduce a new weak choice principle $\mathrm{nRC_{fin}}$:
\vspace{0.4em}\begin{addmargin}[27pt]{27pt}\textit{Given any infinite set $x$, there is an infinite subset $y\subseteq x$ and a selection function $f$ that chooses an $n$-element subset from every finite $z\subseteq y$ containing at least $n$ elements.}
\end{addmargin}\vspace{0.4em}
 By constructing new permutation models built on a set of atoms obtained as Fra\"iss\'e limits, we will study the relation of $\mathrm{nRC_{fin}}$ to the weak choice principles $\mathrm{RC_m}$ (that has already been studied in \cite{lorenz} and \cite{montenegro}): 
\vspace{0.4em} \begin{addmargin}[27pt]{27pt}
 \textit{Given any infinite set $x$, there is an infinite subset $y\subseteq x$ with a choice function $f$ on the family of all $m$-element subsets of $y$.}
\end{addmargin} \vspace{0.4em}
 Moreover, we prove a stronger analogue of the results in \cite{montenegro} when we study the relation between $\mathrm{nRC_{fin}}$ and $\mathrm{kC_{fin}^-}$ which is defined by: 
 \vspace{0.4em}\begin{addmargin}[27pt]{27pt}
 \textit{Given any infinite family $\mathcal{F}$ of finite sets of cardinality greater than $k$, there is an infinite subfamily $\mathcal{A}\subseteq \mathcal{F}$ with a selection function $f$ that chooses a $k$-element subset from each $A\in\mathcal{A}$.}
 \end{addmargin}

%We introduce a new class of weak choice principles $\mathrm{nRC_{fin}}$ and mainly study its relation and similarities with the choice axioms $\mathrm{RC_m}$, object of interest in \cite{montenegro} and \cite{lorenz}. On the one hand, a stronger analogue of \cite{montenegro} is proved. On the other, by means of new permutation models built on sets of atoms obtained as Fraïssé limits, following methods available in \cite{libro}, we give an extensive characterization of the relation between the two mentioned classes.
\end{abstract}
\\ \\
\mathclass{\textbf{03E25} 03E35}

\section{Notation and Choice Principles}
In this paper we will use the following terminology: 
\begin{itemize}
	\item By $\omega$ we denote the set of all natural numbers $\{0,1,2,\dots\}$ and $\mathrm{fin}(\omega)$ denotes the set of finite subsets of $\omega$.
    \item Given a set $x$ and a natural number $n$, $[x]^n$ is defined as the set of \emph{all\/} the subsets of $x$ with cardinality $n$. Similarly, $[x]^{>n}$ is the set of all the \emph{finite\/} subsets of $x$ with cardinality greater than $n$.
    \item Given a permutation model $\mathcal{M}$ and a statement $\phi$, we will write $\mathcal{M}\models\phi$ to indicate that $\phi$ holds in $\mathcal{M}$.
    \item $\mathrm{BFM}$ is the well known Basic Fraenkel Model.
\end{itemize}

Furthermore, we shall use the following notation for weak choice principles:
\begin{itemize}
    \item $\mathrm{RC_n}$ is the following axiom: given any infinite set $x$, there exists an infinite subset $y\subseteq x$ with a choice function $f\colon[y]^n\to y$ such that, for all $z\in[y]^n$, $f(z)\in z$.
        \item $\mathrm{nRC_{fin}}$ is the following axiom: given any infinite set $x$, there exists an infinite subset $y\subseteq x$ with a selection function $f\colon[y]^{>n}\to[y]^n$ such that, for all $z\in[y]^{>n}$, $f(z)\subseteq z$.
    \item $\mathrm{C_n}$ is the following axiom: any infinite family $\mathcal{A}$ of sets of cardinality $n$ has a choice function $f\colon\mathcal{A}\to\bigcup\mathcal{A}$ such that, for all $A\in\mathcal{A}$, $f(A)\in A$.
    \item $\mathrm{C^-_n}$ is the following axiom: given any infinite family $\mathcal{F}$ of non-empty sets with cardinality $n$, there exists an infinite subfamily $\mathcal{A}\subseteq\mathcal{F}$ which has a choice function $f\colon\mathcal{A}\to\bigcup\mathcal{A}$ such that, for all $A\in\mathcal{A}$, $f(A)\in A$.
    \item $\mathrm{nC_{fin}}$ is the following axiom: any infinite family $\mathcal{A}$ of finite sets with cardinality greater than $n$ has a selection function $f\colon\mathcal{A}\to[\bigcup\mathcal{A}]^n$ such that, for all $A\in\mathcal{A}$, $f(A)\subseteq A$.
    \item $\mathrm{nC^-_{fin}}$ is the following axiom: given any infinite family $\mathcal{F}$ of finite sets with cardinality greater than $n$, there exists an infinite subfamily $\mathcal{A}\subseteq\mathcal{F}$ which has a selection function $f\colon\mathcal{A}\to[\bigcup\mathcal{A}]^n$ such that, for all $A\in\mathcal{A}$, $f(A)\subseteq A$.
    \item $\mathrm{ACF^-}$ is the following axiom: given any infinite family $\mathcal{F}$ of non-empty finite sets, there exists an infinite subfamily $\mathcal{A}\subseteq\mathcal{F}$ which has a choice function $f\colon\mathcal{A}\to\bigcup\mathcal{A}$ such that, for all $A\in\mathcal{A}$, $f(A)\in A$.
\end{itemize}

\section{Introduction}
Following the terminology used in \cite{reduced}, we introduce a new class of diminished choice principles $\mathrm{nRC_{fin}}$ and study its relation with the two classes $\mathrm{RC_n}$ and $\mathrm{nC^-_{fin}}$, which have in general inspired the new one; we will indeed obtain analogous results to \cite{lorenz} and \cite{salome}. The title of each section refers to the following diagram. The labels of the arrows indicate in which section we analyze that specific implication. Here, n,k,m,j stand all for natural numbers.

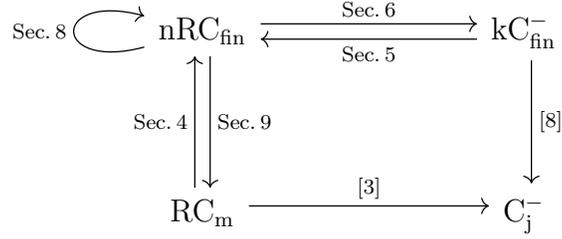
\begin{center}
\begin{tikzcd}[column sep=7em, row sep=4em]
     \mathrm{nRC_{fin}} 
     \arrow [loop left, lu, "\mathrm{Sec.\,8}"]
     \arrow[r, shift left, "\mathrm{Sec.\,6}"]
     \arrow[d, shift left, "\mathrm{Sec.\,9}"]
     & \mathrm{kC^-_{fin}} \arrow[shift left, l, "\mathrm{Sec.\,5}"]
     \arrow[d, shift left, "\scriptsize{\cite{salome}}"]\\
     \mathrm{RC_m} \arrow[shift left, u, "\mathrm{Sec.\,4}"]
     \arrow[r, shift left, "\scriptsize{\cite{lorenz}}"]
     & \mathrm{C^-_j}  
\end{tikzcd}
\end{center}

\noindent To be more precise, we will prove the following results:
\begin{itemize}
	\item Relation between $\mathrm{nRC_{fin}}$ and $\mathrm{RC_m}$:
	\begin{itemize}
		\item For each $n\in\omega$, $\mathrm{RC_n}\nRightarrow\mathrm{nRC_{fin}}$ in ZF+$\neg$AC.
		\item For all $k,n\in\omega$, $\mathrm{nRC_{fin}}\Rightarrow \mathrm{RC_{kn+1}}$.
		\item $\mathrm{4RC_{fin}}\Rightarrow \mathrm{RC_n}$ whenever $n$ is odd and greater than $4$.
	\end{itemize}
	
	\item Relation between $\mathrm{nRC_{fin}}$ and $\mathrm{kC_{fin}^-}$:
	\begin{itemize}
		\item For each $n\in\omega$, $\mathrm{nC_{fin}^-}\nRightarrow\mathrm{nRC_{fin}}$ in ZF+$\neg$AC.
		\item For all $n\in\{2,3,4,6\}$, $\mathrm{nRC_{fin}}\Rightarrow \mathrm{nC_{fin}^-}$.
		\item For all primes $p$ and all $k\in\omega$ we have that $\mathrm{p^kRC_{fin}}\Rightarrow \mathrm{p^kWOC_{fin}^-}$.
	\end{itemize}
	
	\item A relation between $\mathrm{nRC_{fin}}$ and $\mathrm{C_k^-}$:
	\begin{itemize}
		\item $\mathrm{4RC_{fin}}\Rightarrow \mathrm{C_3^-}$.
	\end{itemize}
	
	\item Relation between $\mathrm{nRC_{fin}}$ and $\mathrm{kRC_{fin}}$:
	\begin{itemize}
		\item For all $k,n\in\omega$, $\mathrm{nRC_{fin}}\Rightarrow \mathrm{knRC_{fin}}$.
		\item Let $k,n\in \omega$ with $k>n$. If $k$ is not a multiple of $n$, then
		$\mathrm{nRC_{fin}}\nRightarrow\mathrm{kRC_{fin}}$ in ZF+$\neg$AC.
	\end{itemize}
\end{itemize}

\section{Approach and Transferability}

We will prove independence between choice principles in $\mathrm{ZF}$ via permutation models. In a few words, we can say that a permutation model is built from a ground model, which is a model of $\mathrm{ZFA}$: a variation of $\mathrm{ZF}$ set theory in which the axiom of extensionality is weakened in order to allow the existence of new objects (called atoms) containing no elements, but which are still distinct from the empty set. From this ground model (which satisfies $\mathrm{AC}$), one can extract a submodel of $\mathrm{ZFA}$ in which $\mathrm{AC}$ fails. For details regarding this construction, see, for example, \cite{libro}. We will simply denote a permutation model by the structure of the set of atoms $A$, the normal ideal $I$ on $A$ and the group of permutations $G$: the normal filter on $G$ will always be the one generated by $I$. Given a permutation model, we will get conclusion regarding $\mathrm{ZF}$ in the following way: Suppose we manage to build a permutation model in which a certain choice principle $\mathrm{Ax1}$ holds and some other $\mathrm{Ax2}$ fails. Using the results of \cite{pincus}, we can conclude that if $\mathrm{Ax1}$ and $\mathrm{Ax2}$ both belong to a certain class of statements (in which case the statements are said to be injectively boundable), then there is a model of $\mathrm{ZF}$ in which $\mathrm{Ax1}$ holds, $\mathrm{Ax2}$ fails, and both, $\mathrm{Ax1}$ and $\mathrm{Ax2}$, have the same $\textit{meaning}$ as in the permutation model, i.e., cardinalities and cofinalities remain unchanged between the two models. For the definition of injectively boundable, see \cite{pincus} or \cite{libro2}. Once done, it is not hard to see that all the choice principles we will consider are injectively boundable:  an injection of $\omega$ in an infinite set gives an infinite subset of which the power set admits a choice function.  

\section{Vertical Upward}

In this section we show that for any positive $m\in\omega$, $(\forall n\in\omega\:\mathrm{RC}_n)$ does not imply $m\mathrm{RC}_\mathrm{fin}$. To this end, we use the model which in \cite{salome} is called $\mathcal{V}_\mathrm{fin}$. The results contained in this section are not new and can be found stated in \cite{libro2} and proved in \cite{levy}.

The model $\mathcal{V}_\mathrm{fin}$ is constructed from a countable set of atoms $A$ partitioned in a well ordered family of blocks $\{B_i:i\in\omega\}$, such that for every $i\in\omega$, $B_i$ has cardinality $p_i$, where $p_i$ is the $i$-th prime number. For each $i\in\omega$, fix a cyclic permutation $\varphi_i$ on $B_i$ that has no fixed points. The considered group of permutations $G$ is given by all the permutations $\varphi$ on $A$ that move only finitely many atoms and such that, for every $i\in\omega$, $\varphi$ restricted to $B_i$ equals some power of $\varphi_i$. The corresponding normal filter is generated by the normal ideal of all finite subsets of $A$.

\begin{theorem}
We have that $\mathcal{V}_\mathrm{fin}\models\forall n\in\omega\:(\mathrm{C_n}\land\lnot\mathrm{nRC_{fin}})$.
\end{theorem}

Since evidently $\mathrm{C_n}$ implies $\mathrm{RC_n}$, the theorem proves what we claimed in this section.

\section{Horizontal Left}

In this section we briefly mention that any conjunction of $\mathrm{mC^-_{fin}}$ does not imply any $\mathrm{nRC_{fin}}$.

\begin{theorem}
We have that $\mathrm{BFM}\models\forall n\in\omega\:(\mathrm{nC^-_{fin}}\land\lnot\mathrm{nRC_{fin}})$.
\end{theorem}

\begin{proof}
It is known (see, e.g., \cite{libro2}) that $\mathrm{ACF^-}$ holds in $\mathrm{BFM}$, and it is easy to see that by $n$ consecutive applications of $\mathrm{ACF^-}$ one obtains $\mathrm{nC^-_{fin}}$. The conclusion follows from noticing that any $\mathrm{nRC_{fin}}$ fails on the set of atoms.
\end{proof}

\section{Horizontal Right}

\subsection{Positive}

This subsection starts with the very few cases in which we have a full positive answers.

\begin{lemma}
For $n\in\{2,3\}$ we have $\mathrm{nRC_{fin}}\Rightarrow\mathrm{nC_{fin}^-}$.
\end{lemma}

\begin{proof}
We prove each case separately. Let $n=2$ and $\mathcal{A}=\{A_j:j\in J\}$ be an infinite family of pairwise disjoint finite sets (we can assume they are disjoint by replacing each $A_i$ with the unique function $A_i^*:A_i\to\{A_i\}$). Set $x=\bigcup\mathcal{A}$ and apply $\mathrm{2RC_{fin}}$ to get an infinite $y\subseteq x$ and a function $g:[y]^{>2}\to[y]^2$ such that, for all $Y\in[y]^{>2}$, $g(Y)\subseteq Y$. Since every element of $\mathcal{A}$ is finite, there must be an infinite subset $I$ of $J$ such that, for all $i\in I$, $A_i\cap y\neq\emptyset$. If for infinitely many $i\in I$, $|A_i\cap y|=2$ the claim is obvious, and likewise if $i\in I$, $|A_i\cap y|>2$ for infinitely many $i\in I$, then we are done by defining for each and every such $i\in I$ the function $f:A_i\mapsto g(A_i\cap y)$. If that is not the case, apply a second time $\mathrm{2RC_{fin}}$ to $\bigcup\{A_i:i\in I\}\setminus y$, to get another infinite subset $z\subseteq x$ with $z\cap y=\emptyset$. If, again, $\{i\in I: |A_i\cap z|\geq2\}$ is finite, then we get that $K=\{i\in I: |A_i\cap y|=|A_i\cap z|=1\}$ is infinite, together with the obvious function $f:A_k\mapsto A_k\cap(y\cup z)$, for all $k\in K$.

For the other case we start similarly: let $n=3$ and $\mathcal{A}=\{A_j:j\in J\}$ be an infinite family of pairwise disjoint finite sets. Set $x=\bigcup\mathcal{A}$ and apply $\mathrm{3RC_{fin}}$ to get an infinite $y\subseteq x$ and a function $g:[y]^{>3}\to[y]^3$ such that, for all $Y\in[y]^{>3}$, $g(Y)\subseteq Y$. Since every element of $\mathcal{A}$ is finite, there must be an infinite subset $I$ of $J$ such that, for all $i\in I$, $A_i\cap y\neq\emptyset$. If $\{i\in I: |A_i\cap y|=1\:\mathrm{or}\:|A_i\cap y|\geq 3\}$ is infinite, with a perfectly analogous approach to the previous case we get the conclusion. Otherwise, for all but finitely many $i\in I$, $|A_i\cap y|=2$. At this point we use Montenegro's result that $\mathrm{RC_4}$ implies $\mathrm{C_4^-}$, which in turns implies $\mathrm{C_2^-}$. Since, by Lemma [9.1], $\mathrm{3RC_{fin}}$ implies $\mathrm{RC_4}$, by applying $\mathrm{C_2^-}$ to $\{A_i: |A_i\cap y|=2\}$ we get to a case which has already been solved, namely the one in which $\{i\in I: |A_i\cap y|=1\}$ is infinite.
\end{proof}

In the proofs of the following two theorems we use the ideas Montenegro needed in \cite{montenegro} to show the implication $\mathrm{RC_4}\implies\mathrm{C^-_4}$.
%
%In the two following theorems we make use of the same idea as in Montenegro's proof of $\mathrm{RC_4}\implies\mathrm{C^-_4}$, available in \cite{montenegro}. It has now to be adapted to the case, in contrast with the previous lemma, where it has been possible to apply directly the result.

\begin{theorem}
$\mathrm{4RC_{fin}}\Rightarrow\mathrm{4C_{fin}^-}$.
\end{theorem}

\begin{proof}
Let $\mathcal{A}=\{A_j:j\in J\}$ be an infinite family of pairwise disjoint finite sets. Set $x=\bigcup\mathcal{A}$ and apply $\mathrm{4RC_{fin}}$ to get an infinite $y\subseteq x$ and a function $g\colon[y]^{>4}\to[y]^4$ such that, for all $Y\in[y]^4$, $g(Y)\subseteq Y$. Since every element of $\mathcal{A}$ is finite, there must be an infinite subset $I$ of $J$ such that, for all $i\in I$, $A_i\cap y\neq\emptyset$. With perfectly analogous arguments as in the previous lemma, it is easy to see that the only difficult case is when, for all $i\in I$, $|A_i\cap y|=3$. The following part of the proof shows that $\mathrm{4RC_{fin}}$ implies $\mathrm{C_3^-}$.

For all $i\in I$, set $B_i=A_i\cap y$. We define a directed graph $G\subseteq I^2$ on $I$: let $(i,j)$ be an edge if and only if $B_j\nsubseteq g(B_i\cup B_j)$. The idea behind this definition is that every time $(i,j)$ is an edge, then the function $g$ selects one element from $B_j$ whenever considered together with $B_i$ (we choose the element in $B_i\setminus g(B_i\cup B_j)$ if $\lvert B_i\cup g(B_i\cup B_j)\rvert=2$). We say that an $i\in I$ has outdegree $k$ whenever $|\{j\in I: (i,j)\in G\}|=k$. Notice that we can assume that no $i\in I$ has infinite outdegree, otherwise we could easily select one element from infinitely many $B_i$. Now we claim that for each $k\in\omega$ there are only finitely many $i\in I$ such that $i$ has outdegree $k$. To prove this, assume towards a contradiction that there exists some $k'\in\omega$ and $\widetilde{I}\subseteq I$ such that $|\widetilde{I}|=2k'+3$ and that for all $i\in\widetilde{I}$, $i$ has outdegree $k'$. By construction, if $n$ is the number of edges contained in $\widetilde{I}^2$, then $\binom{|\widetilde{I}|}{2}\leq n\leq|\widetilde{I}|k'$, from which follows $k'+1\leq k'$, a contradiction. We have obtained a well ordered partition of $I$ into finite classes according to the outdegree of every $i\in I$. In symbols: 
$$
I_k=\{i\in I: i\textrm{ has outdegree }k\}\text{ for every }k\in\omega.
$$ 
Applying $\mathrm{4RC_{fin}}$ to $I$, we extract at most $4$ elements from each class, and for some $1\leq m\leq 4$, we get exactly $m$ elements from infinitely many classes, so we can assume that we get $m$ elements from every class. Write $f(I_k)$ for the $m$ extracted elements from $I_k$. We finish the proof by analyzing each of these cases separately.

If $m=1$, then, for all $k\in\omega$, there is at least one edge between the element of $f(I_{2k})$ and the one of $f(I_{2k+1})$, and this allows us to select one element from $B_{2k}$ or $B_{2k+1}$.

If $m=2$, we just consider, for all $k\in\omega$, the edges (there must be at least one) between the two elements of $f(I_k)$, and conclude the proof as in the previous case.

If $m=3$, consider again, for all $k\in\omega$, all the inner edges contained in $f(I_k)^2$. Since each $i\in f(I_k)$ can be chosen at most $2$ times and there are at least $3$ inner edges, we are always able to choose one element from some some $B_j$ with $j\in f(I_k)$.

If $m=4$, consider, for all $k\in\omega$, $g(\cup_{i\in f(I_k)}B_i)$. If it selects less than $4$ elements from $f(I_k)$, we are in one of the previous cases and if it selects exactly one element from each $B_i$,  with $i\in f(I_k)$, we are also done. This concludes the proof.
\end{proof}

In the proof of the following theorem we will also use, without going into details, techniques and arguments which were carefully explained in the two preceding proofs.

\begin{theorem}
$\mathrm{6RC_{fin}}\Rightarrow\mathrm{6C^-_{fin}}$.
\end{theorem}

\begin{proof}
Let $\mathcal{A}=\{A_j:j\in J\}$ be an infinite family of pairwise disjoint finite sets. Set $x=\bigcup\mathcal{A}$ and apply $\mathrm{6RC_{fin}}$ to get an infinite $y\subseteq x$ and a function $g\colon[y]^{>6}\to[y]^6$. As usual, we can assume $|A_j\cap y|<6$ for all $j\in J$. Moreover, it is possible to assume $|A_j\cap y|<4$ for all $j\in J$, as well. To see it, take for instance the case in which $|A_i\cap y|=5$ for all the infinitely many $i\in I\subseteq J$. As in the previous theorem, define an oriented graph $G\subseteq I^2$ and let $(i,j)$ be an edge if and only if $A_j\nsubseteq g(A_i\cup A_j)$. This way, we obtain a well ordered partition of $I$ into finite classes $I_k$, for $k\in\omega$, according to the outdegree of each $i\in I$. Apply $\mathrm{6RC_{fin}}$ to $I$ and extract a finite set $f(I_k)$ of at most $6$ elements from each class $I_k$. Then extract again at most $6$ elements from each $\cup_{i\in f(I_k)}(A_i\cap y)$. The only case which is not solved by the last $\textit{extraction}$ is when $|f(I_k)|=1$ for all $k\in\omega$, but this is easily handled as the case $m=1$, at the end of the previous proof. The case when $|A_i\cap y|=4$ for infinitely many $i\in I\subseteq J$ can be solved in the same way.

Now, given $\mathcal{A}=\{A_j:j\in J\}$ and $y\subseteq\bigcup\mathcal{A}$, let $I=\{i\in J:A_i\cap y\neq\emptyset\}$. Apply $\mathrm{6RC_{fin}}$ to $\cup_{i\in I}A_i\setminus y$ to get an infinite $z\subseteq \bigcup\mathcal{A}\setminus y$. Similarly, if $K=\{k\in I:A_k\cap z\neq\emptyset\}$, apply $\mathrm{6RC_{fin}}$ to $\cup_{k\in K}A_k\setminus(y\cup z)$ to get an infinite $w\subseteq \bigcup\mathcal{A}\setminus (y\cup z)$. A straightforward analysis shows that the only non trivial case is given, modulo symmetries, by the one in which
$$
|A_j\cap y|=3,\,|A_j\cap z|=|A_j\cap w|=2\textrm{ and }|A_j|=7,\textrm{ for all }j\in J.$$ 
Our goal is to select one element either from $|A_j\cap y|$ or $|A_j\cap z|$, for infinitely many $j\in J$. In order to do this, we consider the family of edges $\mathcal{E}=\{E_j\coloneqq(A_j\cap y)\times(A_j\cap z):j\in J\}$ and the corresponding partitions 
$$
F^j_a=\{e\in(A_j\cap y)\times(A_j\cap z):e(1)=a\},\,a\in A_j\cap y,
$$ 
$$
G^j_b=\{e\in(A_j\cap y)\times(A_j\cap z):e(2)=b\},\,b\in A_j\cap z.
$$
Notice that for all $j\in J$, $|F^j_a|=2$ and $|G^j_b|=3$. It is easy to see that whenever we select a proper subset of $E_j$ for some $j\in J$, we are able to select one element from $A_j\cap y$ or from $A_j\cap z$. Also for this reason, when applying $\mathrm{6RC_{fin}}$ to $E\coloneqq\bigcup\mathcal{E}$, we can assume that we get a selection function $f$ on the set of all edges $E$. To simplify the notation, let $\widetilde{f}$ be defined as $\widetilde{f}\colon[E]^7\to E$, $\widetilde{f}\colon S\mapsto S\setminus f(S)$. Now, for $j\in J$ and $b\in A_j\cap z$, define the degree $$\mathrm{deg}(G^j_b)=|\{F^i_a\cup F^i_{a'}:i\in J\wedge a,a'\in(A_j\cap z)\wedge \widetilde{f}(G^j_b\cup F^i_a\cup F^i_{a'})\in F^i_a\cup F^i_{a'}\}|.$$

We can assume that every $G_b^j$ has finite degree, since we would be otherwise able to select a proper subset from infinitely many $E_j$. In addition, assume that for some $k_0\in\omega$ there are infinitely many $G^j_b$ with degree equal to $k_0$. Then order $k_0+1$ distinct $4$-element sets of the form $F^i_a\cup F^i_{a'}$ for some $i\in J$ and $a,a'\in(A_j\cap z)$. For each $G^j_b$, there must be a first of these $k_0+1$ sets with the property that $\widetilde{f}(G^j_b\cup F^i_a\cup F^i_{a'})\in G^j_b$, but this fact allows us to select one edge from each $G^j_b$ with degree equal to $k_0$. Thus, assume that for each $k\in\omega$ there are only finitely many $G^j_b$ with degree $k$. This gives us a well ordered partition of $\{G^j_b:j\in J\wedge b\in A_j\cap z\}$ into finite subclasses. Explicitly into the subclasses
$$
H_k=\{G^j_b:\mathrm{deg}(G^j_b)=k\}.
$$
Apply one last time the function $f$ to each $\bigcup H_k$ and notice that the only case in which we are not able to select a proper subset from infinitely many $E_j$, is when, for all but finitely many $k\in\omega$, $f(\bigcup H_k)=E_i$ for some $i\in J$. We conclude the proof by solving this last case. Suppose that for infinitely many $k\in\omega$, given $f(\bigcup H_k)=E_i$, there is at least one $G^i_b\subseteq E_i$ such that for an $l\in J$ and a $k'\in\omega$, with $f(\bigcup H_{k'})=E_l$, it is possible to select an element from $G^i_b$ by considering the set 
$$
\mathrm{sel}(G_b^i,l)\coloneqq\{\widetilde{f}(G^i_b\cup F^l_a\cup F^l_{a'}):|F^l_a\cup F^l_{a'}|=4\wedge F^l_a\cup F^l_{a'}\subseteq E_l\}.
$$ 
Then we can conclude by choosing, for each such $k\in\omega$, the first $k'\in\omega$ with the mentioned property. If that is not the case, fix $k_1\in\omega$ with $f(H_{k_1})=E_i$ such that for infinitely many $j\in J$ and $k\in\omega$ with $f(H_k)=E_j$ we have that 
$$
|\mathrm{sel}(G_b^i,j)|=3\textrm{ for both }b\in A_i\cap z,
$$ 
and conclude by fixing some $b_0\in A_i\cap z$ and $a_0\in A_i\cap y$. This selects a proper subset from infinitely many $E_j$, namely that unique $F^j_a\cup F^j_{a'}\subseteq E_j$ such that $$\widetilde{f}(G^i_b\cup F^j_a\cup F^j_{a'})=(a_0,b_0).$$
\end{proof}

We conclude the subsection with an example of how it is possible to obtain a weaker implication than $\mathrm{nRC_{fin}}\implies\mathrm{nC^-_{fin}}$ for some infinite class of cases. $\mathrm{p^kWOC^-_{fin}}$ is essentially the same axiom as $\mathrm{p^kC^-_{fin}}$. The only difference is that we require the family of finite sets to be well-ordered.

\begin{theorem}
For all primes $p$ and all natural numbers $k$, $\mathrm{p^kRC_{fin}}\Rightarrow\mathrm{p^kWOC^-_{fin}}$.
\end{theorem}

\begin{proof}
Let $\mathcal{A}=\{A_i:i\in\omega\}$ be a well-ordered family of finite sets such that $|A_i|>p^k$ for all $i\in\omega$. $\mathrm{p^kWOC^-_{fin}}$ is basically obtained by repeated applications of $\mathrm{p^kRC_{fin}}$ to $\bigcup\mathcal{A}$, together with the following two considerations: The first is that if a finite sum $\sum a$ of divisors of $p^k$ is such that $\sum a>p^k$, then it is possible to extract a subsum $\sum a'$ such that $\sum a'=p^k$. The second consideration, which allows us to conclude the proof, is the following: Given a well-ordered family $\mathcal{B}=\{B_i:i\in B\}$ of finite sets of the same size $m\nmid p^k$, with $\mathrm{p^kRC_{fin}}$ we can extract a family of subsets $\mathcal{B'}=\{B'_i:i\in B'\subseteq B\}$ such that for every $i\in B'$, $\emptyset\subsetneq B'_i\subsetneq B_i$. To see this, it is enough to apply $\mathrm{p^kRC_{fin}}$ to $\bigcup\mathcal{B}$ and, if needed, to choose $p^k$ elements from the union the first $l$ sets, where $l$ is the least natural number such that $lm>p^k$, and repeat for every next block of $l$ elements of the family $\mathcal{B}$.
\end{proof}

\subsection{Negative}

A partial negative answer is provided by the models $\mathcal{V}$, introduced and used in \cite{lorenz}, to which we refer for more detailed explanations. In general, the model $\mathcal{V}_n$ has a countable set of atoms $A$ partitioned in blocks $A_i=\{a^i_1,\dots,a^i_n\},i\in\mathbb{Q},$ of size $n$ which are linearly ordered isomorphically to $\mathbb{Q}$. The normal ideal is the one given by the finite subsets and the permutation group $G$ is the one generated by all those permutations $\varphi_i$ on $A$ that act as the identity on $A\setminus A_i$ and as the cycle $(a^i_1,\dots,a^i_n)$ on $A_i$, for some $i\in\mathbb{Q}$. $\mathcal{V}_n$ is generalized to $\mathcal{V}_{n_1,\dots,n_l}$, which is built basically in the same way, but in which the set of atoms is partitioned in $l$ distinct and disjoint $\mathbb{Q}$-lines of blocks. We have the following result.

\begin{theorem}
Let $l\in\omega$, $p_1,\dots,p_l$ be distinct primes and $a_1,\dots,a_l$ natural numbers greater than $0$. Then, we have that $$\mathcal{V}_{p_1^{a_1},\dots,p_l^{a_l}}\models\mathrm{nRC_{fin}}\iff\mathcal{V}_{p_1^{a_1},\dots,p_l^{a_l}}\models\mathrm{nC^-_{fin}}\iff n\textrm{ is a multiple of }\prod_{k=1}^lp_k^{a_k}.$$
\end{theorem}

\begin{proof}
It suffices to prove the theorem for $l=1$. The general case then follows from
$$
\mathcal{V}_{p_1^{a_1},\dots,p_l^{a_l}}\models\mathrm{nRC_{fin}}\iff\bigwedge_{k=1}^l\mathcal{V}_{p_k^{a_k}}\models\mathrm{nRC_{fin}},
$$ 
and 
$$
\mathcal{V}_{p_1^{a_1},\dots,p_l^{a_l}}\models\mathrm{nC^-_{fin}}\iff\bigwedge_{k=1}^l\mathcal{V}_{p_k^{a_k}}\models\mathrm{nC^-_{fin}}.
$$ 
In \cite[Proposition 5.3, Lemma 5.4]{salome}  it is shown that 
$$
\mathcal{V}_{p_1^{a_1}}\models\mathrm{nC^-_{fin}}\iff n\textrm{ is a multiple of }p_1^{a_1}.
$$ 
It remains to show that the same holds for $\mathrm{nRC_{fin}}$. We start with showing that $\mathrm{p_1^{a_1}RC_{fin}}$ holds. In order to do so, we use the construction from \cite[Fact 4]{lorenz} , which we now briefly recall. To help the reader, we use the same notation. Let $x$ be an infinite not well-orderable set with support $E$ and $z\in x$ an element with support $E_z$ which is not supported by $E$. Let $A_r$ be a block of atoms included in $E_z$ but not in $E$. Then, if we define the set $f$ as $$f=\{(\varphi(z),\varphi(A_r)):\varphi\in\mathrm{fix}_G(E_z\setminus A_r)\},$$ the following statements hold:
\begin{itemize}
    \item $f$ is supported by $E_z\setminus A_r$;
    \item $f$ is a function with $\mathrm{dom}(f)\subseteq x$ and $\mathrm{ran}(f)=\{A_q:q\in I\}$ for some possibly unbounded interval $I\subseteq\mathbb{Q}$;
    \item if $y=\mathrm{dom}(f)$ and $\mathcal{Y}=\{f^{-1}(A_q):q\in I\}$, then $\mathcal{Y}$ is a linearly orderable partition of $y$;
    \item the elements of $\mathcal{Y}$ are finite sets all having the same cardinality, which has to be a divisor of $p_1^{a_1}$;
    \item we can write $\mathcal{Y}=\{U_\varphi:\varphi\in\mathrm{fix}_G(E_z\setminus A_r)\}$, where for $\varphi\in\mathrm{fix}_G(E_z\setminus A_r)$, 
    $$
    U_\varphi=\{\eta z:\eta\in\mathrm{fix}_G(E_z\setminus A_r),\varphi^{-1}\eta(A_r)=A_r\}.
    $$
\end{itemize}
Consider now the orbits $O_s=\{\varphi(s):s\in[y]^{>p_1^{a_1}},\varphi\in\mathrm{fix}_G(E_z\setminus A_r)\}$ and write $\mathcal{O}=\{O_s:s\in[y]^{>p_1^{a_1}}\}$. The goal is to show that it is possible to choose for each $O_s$ a subset $\tilde{s}\subsetneq s$ such that $|\tilde{s}|=p_1^{a_1}$ and if $O_s=O_t$ with $\varphi(s)=t$, then $\varphi(\tilde{s})=\tilde{t}.$ Notice that this is equivalent to requiring that every time $\varphi(s)=s$ for some $\varphi\in\mathrm{fix}_G(E_z\setminus A_r)$, then $\varphi(\tilde{s})=\tilde{s}.$ Now, fix an $O_s\in\mathcal{O}$. Notice that if $s$ is a union $s=\bigcup\{U_\varphi:\varphi\in P_s\}$ for some subset $P_s\subseteq\mathrm{fix}_G(E_z\setminus A_r)$ the conclusion is trivial. To deal with the other cases, once more we will fully rely on the fact that if a sum of divisors of $p_1^{a_1}$ is greater than $p_1^{a_1}$, then there is a subsum equal to $p_1^{a_1}$. Indeed, notice that for all $a\in s$, the cardinality of $\{\varphi(a):\varphi\in\mathrm{fix}_G(E_z\setminus A_r),\varphi(s)=s\}$ has to be a divisor of $p_1^{a_1}$. The conclusion is given by the last claim together with the fact that if $\tilde{s}\subseteq s$ is a union of orbits in the form $\{\varphi(a):\varphi\in\mathrm{fix}_G(E_z\setminus A_r),\varphi(s)=s\}$, then $\varphi(s)=s$ implies $\varphi(\tilde{s})=\tilde{s}$.

To finish the proof we have to show that $\mathrm{nRC_{fin}}$ is false in $\mathcal{V}_{p_1^{a_1}}$ whenever $n$ is not a multiple of $p_1^{a_1}$. But this can easily be shown on the set of all atoms.
\end{proof}

\section{Intermezzo: A new model}

Fix a positive integer $n$ and let $\mathcal{L}_n$ be the signature containing an $(m+n)$-place relation symbol $\mathrm{Sel_m}$ for each $m\in\omega$ with $m>n$. Let $\mathrm{T_n}$ be the $\mathcal{L}_n$-theory containing the following axiom schema:\vspace{1em}

\begin{addmargin}[27pt]{27pt}
\textit{For each $m\in\omega$ with $m>n$, we have 
$$\mathrm{Sel_m}(x_1,\ldots,x_m,\,x_1',\ldots,x_n')$$
if and only if the following holds:
\begin{itemize}
\item $\bigwedge_{1\le i<j\le m} x_i\neq x_j\;\wedge \bigwedge_{1\le i<j\le n} x_i'\neq x_j'$
\item For each $1\le j\le n$ there is a $1\le i\le m$ such that $x_j'=x_i$.
\item For any $m$ pairwise distinct elements $x_1,\ldots,x_m$ there are $x_1',\ldots,x_n'$ such
that $\mathrm{Sel_m}(x_1,\ldots,x_m,\,x_1',\ldots,x_n')$.
\item If\/ $\mathrm{Sel_m}(x_1,\ldots,x_m,\,x_1',\ldots,x_n')$ and $\rho$ is a permutation 
of\/ $\{1,\ldots,m\}$, then\/ $\mathrm{Sel_m}(x_{\rho(1)},\ldots,x_{\rho(m)},\,x_1',\ldots,x_n')$ 
\end{itemize}
}
%for all pairwise different $x_1,\dots,x_m$, there exists a unique 
%$m$-element subset  $\{\pi(1),\dots,\pi(m)\}$ of $\{1,\dots,n\}$ 
%such that, whenever $\{\sigma(1),\dots,\sigma(n)\}=\{1,\dots,n\},$ 
%\begin{align*}
%\mathrm{Sel_n}(x_{\sigma(1)},\dots&\,,x_{\sigma(m)},x_{\sigma(m+1)},\dots,x_{\sigma(n)})\iff\\[1.2ex]
%&\{\sigma(1),\dots,\sigma(m)\}=\{\pi(1),\dots,\pi(m)\}\wedge\\[.8ex]
%&\hspace{5ex}\{\sigma(m+1),\dots,\sigma(n)\}=\{1,\dots,n\}\setminus\{\pi(1),\dots,\pi(m)\}.
%\end{align*}
%}
\end{addmargin} \vspace{1em}

\noindent In any model of the theory $\mathrm{T_n}$, the set of all the relations $\mathrm{Sel_m}$ is equivalent to a function $\mathrm{Sel}$ which assigns an $n$-element subset to any finite and big enough set. So, for the sake of simplicity we shall write $\mathrm{Sel}(\{x_1,\ldots,x_m\})=\{x_1',\ldots,x_n'\}$ instead of $\mathrm{Sel_m}(x_1,\ldots,x_m,\,x_1',\ldots,x_n')$.
\newline

\noindent For a model $\mathbf{M}$ of $\mathrm{T_n}$ with domain $M$, we will simply write $M\models\mathrm{T_n}$. Let $$\widetilde{C}=\{M:M\in\mathrm{fin}(\omega)\wedge M\models\mathrm{T_n}\}.$$ Evidently $\widetilde{C}\neq\emptyset$. Partition $\widetilde{C}$ into maximal isomorphism classes and let $C$ be a set of representatives. We proceed with the construction of the set of atoms for our permutation model. The next theorem and its proof are taken from \cite[Ch.\,8]{libro}, with a minor difference which will play an essential role in our work.

\begin{theorem}
For any positive integer $n$ there exists a model $\mathbf{F}\models\mathrm{T_n}$ with domain $\omega$ such that:
\begin{itemize}
    \item Given a non empty $M\in C$, $\mathbf{F}$ admits infinitely many submodels isomorphic to~$M$.
    \item Any isomorphism between two finite submodels of\/ $\mathbf{F}$ can be extended to an automorphism of\/ $\mathbf{F}$.
\end{itemize}
\end{theorem}

\begin{proof}
The construction of $\mathbf{F}$ is by induction on $\omega$. Let $F_0=\emptyset$. $F_0$ is trivially a model of $\mathrm{T_n}$ and, for every element $M$ of $C$ with $|M|\leq0$, $F_0$ contains a submodel isomorphic to $M$. Let $F_s$ be a model of $\mathrm{T_n}$ with a finite initial segment of $\omega$ as domain and such that for every $M\in C$ with $|M|\leq s$, $F_s$ contains a submodel isomorphic to $M$. Let
\begin{itemize}
    \item $\{A_i:i\leq p\}$ be an enumeration of $[F_s]^{\leq n}$,
    \item $\{R_k: k\leq q\}$ be an enumeration of all $M\in C$ with $1\leq|M|\leq s+1$,
    \item $\{j_l:l\leq u\}$ be an enumeration of all the embeddings $j_l:F_s|_{A_i}\xhookrightarrow{} R_k$, where $i\leq p$, $k\leq q$ and $|R_k|=|A_i|+1$.
\end{itemize}
For each $l\leq u$, let $a_l\in\omega$ be the least natural number such that $a_l\notin F_s\cup\{a_{l'}:l'<l\}$. The idea is to add $a_l$ to $F_s$, extending $F_s|_{A_i}$ to a model $F_s|_{A_i}\cup\{a_l\}$ isomorphic to $R_k$, where $j_l:F_s|_{A_i}\xhookrightarrow{} R_k$. Define $F_{s+1}:=F_s\cup\{a_l:l\leq u\}$. 

In \cite{libro}, $F_{s+1}$ is made into a model of $\mathrm{T_n}$ in a non-controlled way, while here we impose the following: Let $\{x_1,\dots,x_{m^\prime}\}$ be a subset of $F_{s+1}$ from which we have not already chosen an $n$-element subset. Suppose $m^{\prime}>n$ and that $i>j$ implies $x_i>x_j$ (recall that $F_{s+1}$ is a subset of $\omega$). Then we simply impose $\mathrm{Sel}(\{x_1,\dots,x_{m^{\prime}}\})=\{x_{m^{\prime}-n+1},\dots,x_{m^{\prime}}\}$.\smallskip

The desired model is finally given by $\mathbf{F}=\bigcup_{s\in\omega}F_s$.\smallskip

\noindent We conclude by showing that every isomorphism between finite submodels can be extended to an automorphism of $\mathbf{F}$. Let $i_0:M_1\to M_2$ be an isomorphism of $\mathrm{T_n}$-models. Let $a_1$ be the least natural number in $\omega\setminus(M_1\cup M_2)$. Then $M_1\cup M_2\cup\{a_1\}$ is contained in some $F_n$ and by construction we can find some $a'_1\in\omega$ such that $\mathbf{F}|_{M_1\cup\{a_1\}}$ is isomorphic to $\mathbf{F}|_{M_2\cup\{a'_1\}}$. Extend $i_0$ to $i_1:M_1\cup\{a_1\}\to M_2\cup\{a'_1\}$ by imposing $i_1(a_1)=a'_1$. Let $a_2$ be the least integer in $\omega\setminus(M_1\cup M_2\cup\{a_1,a'_1\})$ and repeat the process. The desired automorphism of $\mathbf{F}$ is $i=\bigcup_{t\in\omega}i_t$.
\end{proof}

\begin{defi}{\rm 
Let us fix some notations and terminology. The elements of the model~$\mathbf{F}$ above constructed will be the atoms of our permutation model. Since for each atom $a$ there is a unique triple $s,i,k$ such that $F_s|_{A_i}\cup\{a\}$ is isomorphic to $R_k$, each atom $a$ corresponds to a unique embedding $j_a:F_s|_{A_i}\xhookrightarrow{} R_k$. 
We shall call the domain of the embedding $j_a$ the $\emph{ground}$ of $a$. Furthermore, given two atoms $a$ and $b$, we say that $a<b$ in case $a<_\omega b$ according to the natural ordering. Notice that this well ordering of the atoms does not exist in the permutation model.}
\end{defi}

Let $A$ be the domain of the model $\mathbf{F}$ of the theory $\mathrm{T_n}$. Then the permutation model $\mathrm{MOD_n}$ is built as follows: Consider the normal ideal given by all the finite subsets of $A$ and the group of permutations $G$ defined by$$\pi\in G\iff \forall\,X\in[\omega]^{\mathrm{fin}},\pi(\mathrm{Sel}(X))=\mathrm{Sel}(\pi X).$$

\begin{theorem}
\label{thm:7.2}
$\mathrm{MOD_n}\models\mathrm{nRC_{fin}}$.
\end{theorem}

\begin{proof}
Firstly, notice that because for any $m>n$ the function $\mathrm{Sel}$ selects an $n$-element set from each $m$-element
set of atoms, $\mathrm{nRC_{fin}}$ holds in $\mathrm{MOD_n}$ for any infinite set of atoms. So, for an infinite set $X$ in $\mathrm{MOD_n}$, it is enough to construct a bijection between an infinite set of atoms and a subset of $X$\,---\,the function $\mathrm{Sel}$ on the finite sets of atoms will then induce a selection function on the finite subsets of some infinite subset of~$X$.

Let $X$ be an infinite set in $\mathrm{MOD_n}$ with support $S'$. If $X$ is well ordered, the conclusion is trivial, so let $x_0\in X$ be an element not supported by $S'$ and let $S$ be a support of $x_0$ with $S'\subseteq S$. Let $a_0\in S\setminus S'$. If $\mathrm{fix_G}(S\setminus\{a_0\})\subseteq\mathrm{Sym_G}(x_0)$ then $S\setminus\{a_0\}$ is a support of $x_0$, so by iterating the process finitely many times we can assume that there exists a permutation $\tau\in\mathrm{fix_G}(S\setminus\{a_0\})$ such that $\tau(x_0)\neq x_0$. Our conclusion will follow by showing that there is a bijection between an infinite set of atoms and a subset of $X$, namely between $\{\pi(a_0):\pi\in\mathrm{fix_G}(S\setminus\{a_0\})\}$ and $\{\pi(x_0):\pi\in\mathrm{fix_G}(S\setminus\{a_0\})\}$. 

Suppose towards a contradiction that there are two permutations $\sigma,\sigma'\in\mathrm{fix_G}(S\setminus\{a_0\})$ such that $\sigma(x_0)=\sigma'(x_0)$ but $\sigma(a_0)\neq\sigma'(a_0)$. Then, by direct computation, the permutation $\sigma^{-1}\sigma'$ is such that $\sigma^{-1}\sigma'(a_0)\neq a_0$ and $\sigma^{-1}\sigma'(x_0)=x_0$. Let $b=\sigma^{-1}\sigma'(a_0)$. Then $\{b\}\cup (S\setminus\{a_0\})$ is a support of $x$. By construction, the set $\{\pi(a_0):\pi\in\mathrm{fix_G}(\{b\}\cup (S\setminus\{a_0\}))\}$ is infinite, from which we deduce that also the set 
$$
L=\big{\{}a\in A: \exists\,\pi\in\mathrm{fix_G}(S\setminus\{a_0\})\textrm{ such that $\pi(x_0)=x_0$ and $\pi(a_0)=a$}\big{\}}
$$
is infinite. Now, by assumption
there is a permutation $\tau\in\mathrm{fix_G}(S\setminus\{a_0\})$ such that $\tau(x_0)\neq x_0$. Let $y_0:=\tau(x_0)$. Then a standard argument shows that also 
$$
R=\big{\{}a\in A: \exists\,\pi\in\mathrm{fix_G}(S\setminus\{a_0\})\textrm{ such that $\pi(x_0)=y_0$ and $\pi(a_0)=a$}\big{\}}
$$ 
must be infinite.\newline

\noindent First note that in $L$ (and similarly also in $R$) there are infinitely many elements with ground $S\setminus \{a_0\}$. This is because $(S\setminus \{a_0\})\cup \{a_0\}\subseteq \textbf{F}$ is a finite model of $\mathrm{T_n}$ and in the construction of our permutation model we add infinitely many atoms $a_l$ (where from outside, $a_l\in\omega$), such that $(S\setminus \{a_0\})\cup \{a_l\}$ and $(S\setminus \{a_0\})\cup \{a_0\}$ are isomorphic via an isomorphism $\delta$ with $\delta\vert_{S\setminus\{a_0\}}=\operatorname{id}\vert_{S\setminus\{a_0\}}$ and $\delta(a_l)=a_0$. We can extend $\delta$ to an automorphism $\delta\in\mathrm{fix_G}(S\setminus \{a_0\})$. By definition of $L$ we have that $a_l\in L$. \newline

\noindent Let $r\in R$ and $p,l\in L$ all having the same ground $S\setminus \{a_0\}$ such that $r\geq p$, $l\geq p$ and $\min(\{p,q,r\})>\max(S\setminus\{a_0\})$. We want to show that every map
$$
\gamma:(S\setminus \{a_0\})\cup \{p\}\cup\{l\}\to (S\setminus \{a_0\})\cup \{p\}\cup \{r\}
$$
with $\gamma\vert_{(S\setminus \{a_0\})\cup \{p\}}=\operatorname{id}_{(S\setminus \{a_0\})\cup \{p\}}$ and $\gamma(l)=r$ is an isomorphism of $\mathrm{T_n}$-models. Let $X\subseteq (S\setminus \{a_0\})\cup \{p\}\cup\{l\}$. If $\{p,l\}\cap X=\emptyset$ we have that $\gamma(\mathrm{Sel}(X))=\mathrm{Sel}(\gamma(X))$. If $l\in X$ and $p\notin X$ let $\pi_l,\pi_r\in\mathrm{fix_G}(S\setminus\{a_0\})$ with $\pi_l(a_0)=l$ and $\pi_r(a_0)=r$. Then $\pi_r\circ\pi_l^{-1}\vert_X=\gamma\vert_X$. So since $\pi_r\circ\pi_l^{-1}\in G$ we have $\gamma(\mathrm{Sel}(X))=\mathrm{Sel}(\gamma(X))$. In the last case, when $\{p,l\}\subseteq X$, the selection function $n$ biggest elements because of the particular care we took in the construction of the selection function on the set of atoms and since $p,r$ and $l$ have ground $S\setminus \{a_0\}$. So we can extend $\gamma$ to a function $\tau^{\prime}\in\mathrm{fix_G}(\{p\}\cup (S\setminus \{a_0\}))$ with $\tau^{\prime}(l)=r$.\newline

\noindent Let $\pi_r\in\mathrm{fix_G}(S\setminus\{a_0\})$ such that $\pi_r(a_0)=r$ and $\pi_r(x_0)=y_0$. Let $\pi_l\in\mathrm{fix_G}(S\setminus\{a_0\})$ with $\pi_l(a_0)=l$ and $\pi_l(x_0)=x_0$. Then we have that $\pi_r^{-1}\circ\tau^{\prime}\circ\pi_l(a_0)=a_0$ which implies that $\pi_r^{-1}\circ\tau^{\prime}\circ\pi_l(x_0)=x_0$ because the function fixes $S$. So 
\begin{equation}
\label{eq:1}
\tau^{\prime}(x_0)=\tau^{\prime}\circ\pi_{l}(x_0)=\pi_r(x_0)=y_0.
\end{equation}
Now let $\pi_p\in \mathrm{fix_G}(S\setminus\{a_0\})$ with $\pi_p(a_0)=p$ and $\pi_p(x_0)=x_0$. Since $S$ is a support of $x_0$, $\pi_p(S)=\{p\}\cup (S\setminus\{a_0\})$ is also a support of $\pi_p(x_0)=x_0$. Therefore,
$$
\tau^{\prime}(x_0)=x_0.
$$
This is a contradiction to (\ref{eq:1}). So we showed that for all $\sigma,\sigma'\in\mathrm{fix_G}(S\setminus\{a_0\})$, $\sigma(x_0)=\sigma'(x_0)$ implies $\sigma(a_0)=\sigma'(a_0)$, from which we get the desired bijection.

%RICCARDOS VERSION
%Fix an element $o$ of $L\setminus S$. Since $S$ is a support of $L$, $L$ must include the set $P$ of those infinitely many atoms which have been added specifically to have the same role as $o$ when confronted with $S$. With that, we mean that for each $p\in P$ there exists some permutation $\pi_p\in\mathrm{fix_G}(S)$ such that $\pi_p(o)=p$, and that $S$ is the ground of $p$. Only at this point we make use of the particular care we took in the construction of the selection function on the set of the atoms: it is by construction possible to find an $r\in R$ and two $p,l\in P$ such that there exists a permutation $\tau'\in\mathrm{fix_G}(p\cup S\setminus\{a\})$ such that $\tau'(l)=r$ and hence $\tau'(x)=y$, absurd.

\end{proof}

Due to the following theorem, the class of models $\mathrm{MOD_n}$ will not tell us anything about the horizontal implications in the diagram.

\begin{theorem}
For each $n\in\omega$, $\mathrm{MOD_n}\models\mathrm{ACF^-}$.
\end{theorem}

\begin{proof}
Fix $n\in\omega$ and let $\mathcal{A}=\{A_i:i\in I\}$ be a family of finite sets. By applying $\mathrm{nRC_{fin}}$ to $\bigcup\mathcal{A}$, it is enough to show that for all $m\leq n$, $\mathrm{C^-_m}$ holds in $\mathrm{MOD_n}$. Fix $m\in\omega$ with $m\leq n$ and suppose $\mathcal{A}=\{A_i:i\in I\}$ is a family of $m$-element sets, and let $P$ be a support of $\mathcal{A}$. If $\bigcup\mathcal{A}$ is well-orderable we are done, so let $x\in\mathcal{A}$ be an element which is not supported by $P$, let $S'$ be a support of $x$ and $a\in S'\setminus P$ an atom such that for some $\pi\in\mathrm{fix}_G(P\cup (S'\setminus \{a\}))$, we have that $\pi(x)\neq x$, as in the previous proof. Set $S=P\cup (S'\setminus \{a\})$ and $X=\{\pi(x):\pi\in\mathrm{fix}_G(S)\}$. Now, we can replace $\mathcal{A}$ with $\{A_i\cap X:i\in I\}$ since a choice function on this last set gives a choice function on the previous $\mathcal{A}$ as well, and let us assume that $\mathcal{A}$ is  family of $m^\prime$-element sets for some $m^\prime\leq m$. As in the proof of Theorem \ref{thm:7.2} we can show that there is a bijection between the infinite set $X$ and the set of atoms $Y:=\{\pi(a)\mid \pi\in \mathrm{fix}_G(S)\}$. So we can without loss of generality assume that $\mathcal{A}$ is a family of $m^\prime$-element subsets of the atoms. Let $A_i\in\mathcal{A}$ with $A_i\cap S=\emptyset$, let $a_0\in A_i$ and let $R^\prime\subseteq A\setminus (S\cup A_i)$ be an $(n-1)$-element set. By construction of the permutation model, we can find an $r_0\in A\setminus (S\cup A_i\cup R^\prime)$ such that
$$
\forall a\in A_i\setminus \{a_0\}~\left(\mathrm{Sel}(R^\prime\cup\{r_0\}\cup \{a\})=R^\prime\cup \{r_0\}\right)
$$
and 
$$
\mathrm{Sel}(R^\prime\cup\{r_0\}\cup\{a_0\})=R^\prime\cup \{a_0\}.
$$
Define $R:=R^\prime\cup\{r_0\}$. Again by construction of the permutation model, we can find infinitely many $b_0\in A$ that behave the same way as $a_0$ with respect to $R\cup S\cup (A_i\setminus\{a_0\})$. In other words, if $\mathrm{repl}$ is the function that replaces $a_0$ by $b_0$, i.e.
\begin{align*}
\mathrm{repl}:A&\to A\\
x&\mapsto \begin{cases}
a_0 &\text{ if } x=b_0;\\
b_0 & \text{ if } x=a_0;\\
x &\text{ otherwise,}
\end{cases}
\end{align*}
we have that for all $X\subseteq R\cup S\cup (A\setminus\{a_i\})$
\begin{equation}
\label{eq:2}
\mathrm{repl}(\mathrm{Sel}(X\cup\{a_0\}))=\mathrm{Sel}(\mathrm{repl}(X\cup\{a_0\})).
\end{equation}
Define
$$
\gamma: S\cup R\cup A_i\to S\cup R\cup (A_i\setminus\{a_0\})\cup\{b_0\}
$$
by $ \gamma:=\mathrm{repl}\vert_{S\cup R\cup A_i}$. With (\ref{eq:2}) we see that $\gamma$ is an isomorphism of $\mathrm{T_n}$-models because for all $X\subseteq R\cup S\cup A_i$
$$
\gamma(\mathrm{Sel}(X))=\mathrm{Sel}(\gamma(X)).
$$
So we can extend $\gamma$ to the whole model $\mathbf{F}$. Since $\gamma\in \mathrm{fix_G}(S\cup R),$ $\gamma(A_i)\in \mathcal{A}$. So there are infinitely many $A_j\in\mathcal{A}$ such that there is exactly one element $a\in A_j$ with $a\in\mathrm{Sel}(R\cup\{a\})$. Choose this element $a$. This gives a choice function with support $R\cup S$. 

%RICCARDOS VERSION
% Notice that as in the proof of Theorem \ref{thm:7.2} we can show that $X$ is an infinite set. If $|X\cap A_i|=1$ for infinitely many $i\in I$ we are done. Now, we can replace $\mathcal{A}$ with $\{A_i\cap X:i\in I\}$ since a choice function on this last set gives a choice function on the previous $\mathcal{A}$ as well, and let us assume that $\mathcal{A}$ is  family of $m^\prime$-element sets for some $m^\prime\leq m$. Furthermore, notice that $S$ is still a support of $\mathcal{A}$. Fix now a set $R$ of atoms with $|R|=n$ and $R\cap S=\emptyset$. The set $X$ is in correspondence with an infinite set of atoms, and each atom can behave only in two different ways with respect to $R$. We conclude by fixing one of the two ways and by extracting from $\mathcal{A}$ the subfamily in which each $m$-element set, element of the subfamily, contains exactly one element that behave in the fixed way. This gives a selection function with support $S\cup R$. Notice that such $m$-element sets exist and are infinitely many in $\mathcal{A}$ by construction of the relation on the atoms.
\end{proof}

We just mention that fact that all of $\mathrm{C_n}$ and $\mathrm{nC_{fin}}$ for $n\in\omega$ are false in every $\mathrm{MOD_m}$: it is enough to consider the family of all set of atoms of correspondent cardinalities.

\section{Loop}

\subsection{Positive}
In this subsection there is only to notice the straightforward:
\begin{lemma}
\label{lemma}
For all $k,n\in\omega$, $\mathrm{nRC_{fin}}\Rightarrow\mathrm{knRC_{fin}}$.
\end{lemma}

\subsection{Negative}
\begin{theorem}
Let $m,n\in\omega$ with $n>m$. For every $n$ which is not a multiple of $m$, 
 $\mathrm{MOD_m}\not\models\mathrm{nRC_{fin}}$.
\end{theorem}

\begin{proof}
Consider the set of the atoms and suppose that there is an infinite subset $A$ with a function $f$ which selects $n$ elements from every finite and large enough subset of $A$. Let $S$ be a support of $f$. Let $M$ be any model of the theory $\mathrm{T}_m$ with cardinality $|M|=mk$ for $k\in\omega$ such that $m(k-1)<n<mk$. Then it is possible to find an $mk$-element subset $N=\{x_1,\dots,x_{mk}\}\subseteq \omega$ such that:
\begin{enumerate}
    \item $N$ and $M$ are isomorphic as models of $\mathrm{T}_m$;
    \item $\mathrm{Sel}(Z)$ can be fixed arbitrarily whenever $Z\subseteq S\cup N$ with $\lvert Z\cap S\rvert\geq 1$ and $\lvert Z\cap N\rvert\geq 1$;
    \item $\mathrm{Sel}(\{x_{im+1},\dots,x_{mk}\})=\{x_{im+1},\dots,
    x_{(i+1)m}\}$ holds for all $i< k$.
\end{enumerate} 
Notice that that condition 3 is only a matter of reordering. Consider the following permutation of $N$, written as a finite product of finite cycles:
$$
\widetilde{\pi}=\prod_{i< k}(x_{im+1},x_{im+2},\dots,x_{(i+1)m}).
$$
Our conclusion will follow by showing that there is a model $M$ of $\mathrm{T_m}$, a corresponding subset $N\subseteq C$ and a permutation $\pi\in\mathrm{fix_G}(S)$ such that $\pi$ acts on $N$ exactly as $\widetilde{\pi}$ on $M$. First we want to find a $\mathrm{T_m}$-model $M=\{x_1,\dots, x_{mk}\}$ such that $M$ and $\widetilde{\pi}M$ are isomorphic as $\mathrm{T_m}$-models. Naturally we first impose condition 3, namely for all $i\leq k$
$$
\mathrm{Sel}(\{x_{im+1},x_{im+2},\dots,x_{mk}\})=\{x_{im+1},x_{im+2},\dots,x_{(i+1)m}\}.
$$
The main ide of the proof is the following: Let $L$ be a subset of $M$ with $|L|>m$ and $L\neq\{x_{im+1},x_{im+2},\dots,x_{mk}\}$ for every $i\in k$. Consider the orbit $\{\widetilde{\pi}^lL:l\in\omega\}$. Now we choose an $m$-element subset $L^\prime\subseteq L$ and define $\mathrm{Sel}(L):=L^\prime$. Extend this choice to the whole orbit by defining 
$$
\mathrm{Sel}(\widetilde{\pi}^lL):=\widetilde{\pi}^l(\mathrm{Sel}(L)).
$$
The choice of $\mathrm{Sel}(L)$ has to be suitable in the sense that $\widetilde{\pi}^jL=L$ must imply $\widetilde{\pi}^j(\mathrm{Sel}(L))=\mathrm{Sel}(L)$. 
\begin{itemize}
    \item First of all assume that for some $I\subseteq k$, 
    $$
    |L\cap(\bigcup_{i\in I}\{x_{im+1},\dots,x_{(i+1)m}\})|=m.
    $$ 
    Then a suitable choice for $\mathrm{Sel}(L)$ is given by $\bigcup_{i\in I}\{x_{im+1},\dots,x_{(i+1)m}\}\cap L$.
    \item Otherwise, let $J\subseteq k$ be the set of indices $j$ such that $\widetilde{\pi}^s$ fixes 
    $$
    L\cap\{x_{jm+1},\dots,x_{(j+1)m}\}
    $$ 
    only if $s$ is a multiple of $m$. If $|L\setminus\bigcup_{j\in J}\{x_{jm+1},\dots,x_{(j+1)m}\}|\leq m$, then a suitable choice for $\mathrm{Sel}(L)$ is given by any $m$-element subset of $L$ which includes $L\setminus\bigcup_{j\in J}\{x_{jm+1},\dots,x_{(j+1)m}\}$. 
    \item Let $J\subseteq k$ be as above and suppose that $m <|L\setminus\bigcup_{j\in J}\{x_{im+1},\dots,x_{(i+1)m}\}|$. By replacing $L$ by $L\setminus\bigcup_{j\in J} \{x_{im+1},\dots, x_{(i+1)m}\}$ we can assume that for each $i < k$ there exists a $1< s< m$ such that $\widetilde{\pi}^s$ fixes $L\cap\{x_{im+1},\dots,x_{(i+1)m}\}$. Our goal is now to get rid of the case in which, for some $i< k$, $0\neq|L\cap\{x_{im+1},\dots,x_{(i+1)m}\}|\nmid m$. Fix such an $i'< k$ and let $s'\in\omega$ be the least integer greater than $1$ for which $\widetilde{\pi}^{s'}$ fixes $L\cap\{x_{i'm+1},\dots,x_{(i'+1)m}\}$. Then the cardinality $|L\cap\{x_{i'm+1},\dots,x_{(i'+1)m}\}|$ must be a multiple of $\frac{m}{s'}$. Indeed, $\frac{m}{s'}$ is the cardinality of each orbit 
    $$
    \{(\widetilde{\pi}^{s'})^s(x):x\in L\cap\{x_{i'm+1},\dots,x_{(i'+1)m}\}\land s\in\omega\}.
    $$ 
    In the next step we can consider each of these orbits as different subsets of the form $L\cap \{x_{im+1},\dots, x_{(i+1)m}\}$. So we can without loss of generality assume that  $\lvert L\cap \{x_{im+1},\dots, x_{(i+1)m}\}\rvert$ divides $m$ for all $i<k$ and that $\widetilde{\pi}^s$ fixes $L\cap \{x_{im+1},\dots, x_{(i+1)m}\}$ for some $1<s<m$.
    
    \item Finally choose $K\subseteq k$ such that Let finally $J\subseteq k$ be such that
    \begin{enumerate}
        \item $|L\cap(\bigcup_{j\in K}\{x_{jm+1},\dots,x_{(j+1)m}\})|\geq m$ is minimal and 
        \item $|L\cap\{x_{jm+1},\dots,x_{(j+1)m}\}|\mid m$ for all $j\in K$.
    \end{enumerate}
Replace $L$ by $L\cap\left(\bigcup_{j\in K}\{x_{jm+1},\dots,x_{(j+1)m}\}\right)$. Set $a_j=|L\cap\{x_{jm+1},\dots,x_{(j+1)m}\}|$ for each $j\in K$. By writing $\sum_{j\in K}{a_j}=m+(|L|-m)$, we can see that $\gcd_{j\in K}(a_j)\mid (|L|-m)$. Now, notice that in order for a power $\widetilde{\pi}^s$ to fix $L\cap\{x_{jm+1},\dots,x_{(j+1)m}\}$ for some $j\in K$, $s$ has to be a multiple of $\frac{m}{a_j}$. It follows that, in order for a power $\widetilde{\pi}^s$ to fix $L$, $s$ has to be a multiple of $m'=\mathrm{lcm}_{j\in K}(\frac{m}{a_j})=\frac{m}{\gcd_{j\in K}(a_j)}$. Summarizing:
\begin{enumerate}
    \item $\gcd_{j\in K}(a_j)\mid (|L|-m)$.
    \item $\widetilde{\pi}^s$ fixes $L$ if and only if $s$ is a multiple of $m'=\frac{m}{\gcd_{j\in K}(a_j)}$.
    \item $|L|-m<a_j$, for all $j\in K$.
\end{enumerate}
Fix a $j\in K$. The conclusion will follow by finding an $F\subseteq L\cap\{x_{jm+1},\dots,x_{(j+1)m}\}$ of cardinality $|L|-m$ such that whenever some $\widetilde{\pi}^s$ fixes $L$, then $\widetilde{\pi}^s$ fixes $F$ as well. We can find such a set $F$ through the following procedure: Start with $F=\emptyset$. Let $x\in (L\cap\{x_{jm+1},\dots,x_{(j+1)m}\})\setminus F$, and replace $F$ by $F\cup\{(\widetilde{\pi}^{m'})^t(x):t\in\omega\}$, noticing that the cardinality of the orbit is exactly $\gcd_{j\in K}(a_j)$. If $|F|=|L|-m$ we are done, otherwise repeat the procedure with some $y\in (L\cap\{x_{jm+1},\dots,x_{(j+1)m}\})\setminus F$. After a finite number of repetitions we get $\lvert F\rvert=\lvert L\rvert-m$.
\end{itemize}
Now we can show that $S$ is not a support of the selection function $f$ we chose at the beginning of the proof. Let $M$ be the $\mathrm{T_m}$-model we constructed above that satisfies $\widetilde{\pi}M=M$. Let $N\subseteq \omega$ be a $\mathrm{T_m}$-model that is isomorphic to $M$ and satisfies conditions 1,2 and 3. The proof above shows that $\pi(\mathrm{Sel}(L))=\mathrm{Sel}(\pi(L))$ for all $L\subseteq N$. Moreover, condition 2 says that $N$ can even be chosen such that  $\pi(\mathrm{Sel}(L))=\mathrm{Sel}(\pi(L))$ for all $L\subseteq N\cup S$. So $\pi$ can be extended to a function $\pi\in \mathrm{fix_G}(S)$ on the whole model $\mathbf{F}$. Note that for all $n$-element subsets of $N$ we have that $\pi(N)\not = N$. So $S$ is indeed not a support of the selection function $f$. This is a contradiction.
\end{proof}

\section{Vertical Downward}

\subsection{Positive}
As an immediate consequence of Lemma \ref{lemma}, we get the following.
\begin{lemma}
\label{lem:9.1}
For all $k,n\in\omega$, $\mathrm{nRC_{fin}}\Rightarrow\mathrm{RC}_{kn+1}$.
\end{lemma}

It is interesting to notice that Lemma \ref{lem:9.1} and the next theorem are here proven using qualitatively the same approach. Despite this fact, the forthcoming proof is more complex than the other.

\begin{theorem}
\label{thm:9.2}
$\mathrm{4RC_{fin}}\Rightarrow\mathrm{RC_7}$.
\end{theorem}

\begin{proof}
Let $A$ be an infinite set and apply $\mathrm{4RC_{fin}}$ to get an infinite subset $B\subseteq A$ with a function $\widetilde{f}\colon[B]^{>4}_\mathrm{fin}\to[B]^4$. Let $S$ be a $7$-element subset of $x$. In this proof we are going to consider all the possible ways in which the function $\widetilde{f}$ can act on the subsets of $S$ in order to show that it is always possible to choose a particular element of $S$, and hence to verify $\mathrm{RC_7}$. Though making use of symmetries in a few passages, it will substantially be a case-by-case analysis. Let $S=\{x,y,z,a,b,c,d\}$ with $\widetilde{f}(S)=\{a,b,c,d\}$. To simplify the notation, define the two functions
\begin{itemize}
    \item $f\colon[S]^{>4}\to\mathscr{P}(S)$ given by $f\colon T\mapsto T\setminus\widetilde{f}(T)$;
    \item $g\colon[S]^{<3}\to [S]^{<3}$ given by $g\colon T\mapsto f(S\setminus T)$.
\end{itemize}
For simplicity we will write, for instance, $g(a)$ instead of $g(\{a\})$.
We can assume that for all $l\in \widetilde{f}(S)=\{a,b,c,d\}$ we have that $g(l)\cap \{x,y,z\}=\emptyset$. Otherwise, there is a natural way to choose an element from $S$. Now we build, step by step, all the possibilities for $\{g(a),g(b),g(c),g(d)\}$ which do not allow us to immediately choose an element from $S$.
\begin{enumerate}
    \item By symmetry, we can fix $g(d)=f(x,y,z,a,b,c)=\{a,b\}$.
    \item There are now only two non-equivalent cases:\begin{enumerate}
        \item $g(d)=\{a,b\}$ and $g(c)=\{a,b\}$;
        \item $g(d)=\{a,b\}$ and $g(c)=\{a,d\}$, which is equivalent to the third possible choice $g(c)=\{b,d\}$.
    \end{enumerate}
    \item The two cases branch now in five:
    \begin{enumerate}
        \item $g(d)=\{a,b\}$, $g(c)=\{a,b\}$ and $g(b)=\{a,c\}$. This is symmetric to  which is symmetric to $g(d)=\{a,b\}$, $g(c)=\{a,b\}$ and $g(b)=\{a,d\}$.
        \item $g(d)=\{a,b\}$, $g(c)=\{a,b\}$ and $g(b)=\{c,d\}$.
        \item $g(d)=\{a,b\}$, $g(c)=\{a,d\}$ and $g(b)=\{a,c\}$.
        \item $g(d)=\{a,b\}$, $g(c)=\{a,d\}$ and $g(b)=\{c,d\}$.
        \item $g(d)=\{a,b\}$, $g(c)=\{a,d\}$ and $g(b)=\{a,d\}$.
    \end{enumerate}
\end{enumerate}
Notice how the option 3.c can be ignored, since it allows us to choose $a$ in $S$ independently from $g(a)$. With similar arguments we can show that the only four non-symmetric choices for $g(a)$ in which we cannot immediately choose an element from $S$ are:
\begin{enumerate}
    \item $g(d)=\{a,b\}$, $g(c)=\{a,b\}$, $g(b)=\{a,c\}$ and $g(a)=\{b,d\}$;
    \item $g(d)=\{a,b\}$, $g(c)=\{a,b\}$, $g(b)=\{c,d\}$ and $g(a)=\{c,d\}$;
    \item $g(d)=\{a,b\}$, $g(c)=\{a,d\}$, $g(b)=\{c,d\}$ and $g(a)=\{b,c\}$;
    \item $g(d)=\{a,b\}$, $g(c)=\{a,d\}$, $g(b)=\{a,d\}$ and $g(a)=\{c,d\}$.
\end{enumerate}
For each of the above cases we can check that the only permutations on $\{a,b,c,d\}$ that preserve $g$ are given by
\begin{enumerate}
    \item $(a,b)(c,d)$;
    \item $(a,b)$, $(c,d)$, $(a,b)(c,d)$, $(a,c)(b,d)$, $(a,d)(b,c)$;
    \item $(a,c)(b,d)$;
    \item $(a,d)(b,c)$.
\end{enumerate}
In each of these cases, it is possible to select a particular double transposition (in case 2, pick $(a,b)(c,d)$). Last, consider how $g$ acts on the six distinct pairs included in $\{a,b,c,d\}$. A double transposition selects exactly two of these pairs: for instance $(a,b)(c,d)$ selects $\{a,b\}$ and $\{c,d\}$. We conclude the proof by considering the uniquely determined $g(g(a,b)\cup g(c,d))$.
\end{proof}

\begin{corollary}
$\mathrm{4RC_{fin}}$ implies $\mathrm{RC_n}$ whenever $n$ is odd and greater than $4$.
\end{corollary}

\begin{proof}
We have that either $n=1+4k$ or $n=3+4k$ for a $k\in\omega$. The first case follows directly by Lemma \ref{lem:9.1}. In the second case let $x$ be an infinite set and apply $\mathrm{4RC_{fin}}$ to get an infinite subset $y\subseteq x$ with a selection function $f:[y]^{>4}_{\mathrm{fin}}\to [y]^4$. Let $z\subseteq y$ be an $n$-element subset. Apply $f$ exactly $k$ times to find a $3$-element subset $z_0$ of $z$. Then $\lvert z_0\cup f(z)\rvert=7$ and we can use Theorem \ref{thm:9.2}.
\end{proof}

\subsection{Negative}

\begin{theorem}
Let\/ $m,n\in\omega\setminus \{0\}$. Then\/ $\mathrm{MOD_m}\not\models\mathrm{RC_n}$ whenever for some prime $p$ divisor of $m$, $n\not\equiv 1\Mod{p}$.
\end{theorem}

\begin{proof}
Let $n,m\in\omega\setminus\{0\}$ and let $p$ be a prime divisor of $m$ such that $n\not\equiv 1 \Mod{p}$. Consider the set of the atoms and suppose that there is an infinite subset $A$ with a function $f$ which selects an element from every $n$-element subset of $A$. Let $S$ be a support of $f$. %Define an equivalence relation $R$ between the elements of $A\setminus S$, namely $(a,b)\in R$ if and only if there exists a permutation $\pi\in\mathrm{fix_G}(S)$ such that $\pi(a)=b$. Let $C$ an infinite equivalence class of $R$, which has to exist since by construction there are only finitely many equivalence classes. The set $C$ is an infinite set in our permutation model, in general with support $S$. Thus, as every infinite set of atoms, it is by construction $rich$ in the following sense. 
Let $M$ be any $\mathrm{T_m}$-model with cardinality $|M|=n$ and write $n=pk+r$ for unique $k,r\in\omega$, with $1< r< p$. Then it is possible to find an $n$-element subset $N=\{x_1,\dots,x_n\}$ of $C$ such that:
\begin{enumerate}
    \item $N$ and $M$ are isomorphic as models of $T$;
    \item $\mathrm{Sel}(Z)$ can be arbitrarily fixed whenever $Z\subseteq S\cup N$ with $\lvert Z\cap S\rvert\geq 1$ and $\lvert Z\cap N\rvert\geq 1$;
    \item $\mathrm{Sel}(\{x_{im+1},\dots,x_n\})=\{x_{im+1},\dots,
    x_{(i+1)m}\}$ holds for all $i< k$.
\end{enumerate} 
Notice that that condition 3 is only a matter of reordering. Consider the following permutation of $N$, written as finite product of finite cycles. $$\widetilde{\pi}=(x_{pk+1},x_{pk+2},\dots,x_n)\prod_{i=0}^{k-1}(x_{pi+1},x_{pi+2},\dots,x_{p(i+1)})$$
Our conclusion will follow by showing that there is a model $M$ of $\mathrm{T}_m$, a corresponding subset $N\subseteq C$ and a permutation $\pi\in\mathrm{fix_G}(S)$ such that $\pi$ acts on $N$ exactly as $\widetilde{\pi}$ acts on $M$. Notice that every cycle in the definition of $\widetilde{\pi}$ is non trivial if and only if $r\neq 1$. First we want to find a $\mathrm{T_m}$-model $M=\{x_1,\dots, x_n\}$ such that $M$ and $\widetilde{\pi}M$ are isomorphic as $\mathrm{T_m}$-models. Naturally we first impose condition 3, namely that for all $i<k$
$$
\mathrm{Sel}(\{x_{im+1},x_{im+2},\dots,x_n\})=\{x_{im+1},x_{im+2},\dots,x_{(i+1)m}\}.
$$
The main idea of the proof is the following: Let $L$ be a subset of $M$ with $|L|>m$, $L\neq\{x_{im+1},x_{im+2},\dots,x_n\}$ for every $i< k$. Consider the orbit $\{\widetilde{\pi}^lL:l\in\omega\}$. Now we will choose an $m$-element subset $L^\prime\subseteq L$ and define $\mathrm{Sel}(L):=L^\prime$. Extend this choice to the whole orbit by defining 
$$
\mathrm{Sel}(\widetilde{\pi}^lL)=\widetilde{\pi}^l(\mathrm{Sel}(L)).
$$ 
The choice of $\mathrm{Sel}(L)$ has to be suitable in the sense that $\widetilde{\pi}^jL=L$ must imply $\widetilde{\pi}^j(\mathrm{Sel}(L))=\mathrm{Sel}(L)$. 
\begin{itemize}
    \item First of all assume that for some $I\subseteq k$, 
    $$
    |L\cap(\bigcup_{i\in I}\{x_{pi+1},\dots,x_{p(i+1)}\})|\in\{m,m-|L\cap\{x_{pk+1},\dots,x_n\}|\}.
    $$
     Then a suitable choice for $\mathrm{Sel}(L)$ is given by either $L\cap(\bigcup_{i\in I}\{x_{pi+1},\dots,x_{p(i+1)}\})$ or by $L\cap(\bigcup_{i\in I}\{x_{pi+1},\dots,x_{p(i+1)}\}\cup\{x_{pk+1},\dots,x_n\})$.
    \item Otherwise, let $J\subseteq k$ be the set of indices $j$ such that $0<|L\cap\{x_{jp+1},\dots,x_{(j+1)p}\}|\neq p$. Moreover, replace $J$ by $J\cup\{k\}$ if $|L\cap\{x_{kp+1},\dots,x_{n}\}|$ either is $1$ or does not divide $r$. For the sake of notation, let us write $\{x_{kp+1},\dots,x_{(k+1)p}\}$ instead of $\{x_{kp+1},\dots,x_n\}$. If $|L\setminus\bigcup_{j\in J}\{x_{jp+1},\dots,x_{(j+1)p}\}|\leq m$, then we claim that a suitable choice for $\mathrm{Sel}(L)$ is given by any $m$-subset of $L$ which includes $L\setminus\bigcup_{j\in J}\{x_{jp+1},\dots,x_{(j+1)p}\}$. The claim follows by the fact that, given a set $\{y_1,\dots,y_{p'}\}$ for some prime $p'\in\omega$, if $\tau$ is the permutation $(y_1,\dots,y_{p'})$ and some power $\tau^a$ fixes a proper subset $H\subsetneq\{y_1,\dots,y_{p'}\}$, then $\tau^a$ is the identity on $\{y_1,\dots,y_{p'}\}$.
\end{itemize}
Note that we covered every possible case. Indeed, if we are not in the last case, then for some $k',r'\in\omega$ with $r'\leq r$ it is true that $m< k'p+r'$. Then, since $r< p$ and $p\mid m$, we are  actually in the first case. 

Now we can show that $S$ is not a support of the selection function $f$ we chose at the beginning of the proof. Let $M$ be the $\mathrm{T_m}$-model we constructed above that satisfies $\widetilde{\pi}M=M$. Let $N\subseteq \omega$ be a $\mathrm{T_m}$-model that is isomorphic to $M$ and satisfies conditions 1,2 and 3. With the proof above and condition 2 we can choose $N$ such that $\pi(\mathrm{Sel}(L))=\mathrm{Sel}(\pi(L))$ for all $L\subseteq N\cup S$. So $\pi$ can be extended to a function $\pi\in \mathrm{fix_G}(S)$ on the whole model $\mathbf{F}$. Note that for all $n$-element subsets of $N$ we have that $\pi(N)\not = N$. So $S$ is indeed not a support of the selection function $f$. This is a contradiction.
\end{proof}
With the same arguments it is possible to emulate the previous result in the following way.
\begin{theorem}
Let $m\in\omega$ be greater than $2$. Then for all $1< n< m$, $\mathrm{MOD_m}\not\models\mathrm{RC_n}$.
\end{theorem}

\begin{proof}
Exactly as in the previous theorem: just consider the permutation 
$$
\widetilde{\pi}=(x_1,\dots,x_n)
$$
and impose that $\mathrm{Sel}(L)\supset L\cap N$ whenever $L\cap N\neq\emptyset$, with $L\subseteq S\cup N$.
\end{proof}

As an immediate consequence of the last results, we get the following Corollary:
\begin{corollary}
Let $k\in\omega\setminus\{0\}$, $\{p_1,\dots,p_k\}$ be distinct prime numbers and $n=\prod_{i=1}^{k}p_i$. Then\/ $\mathrm{nRC_{fin}}\Rightarrow\mathrm{RC_m}$ if and only if\/ $m\equiv 1\Mod{n}$.
\end{corollary}

\section{Open Questions}

\begin{itemize}
	\item For $n\in\{2,3,4,6\}$ we have that $\mathrm{nRC_{fin}}\Rightarrow \mathrm{nC_{fin}^-}$. Does this implication hold for $n=5$? Or more generally: For which $n\in\omega$ does this implication hold?
	\item Write a natural number as unique product of powers of primes $n=\prod_{i=1}^{k}p_i^{m_i}$. Is it the case that\/ $\mathrm{nRC_{fin}}\Rightarrow\mathrm{RC_m}$ if and only if\/ $m>n$ and $m\equiv1\Mod{\prod_{i=1}^{k}p_i}$?
\end{itemize}

\end{document}